\documentclass[a4paper,reqno]{amsart}

\usepackage{amssymb}
\usepackage{latexsym}
\usepackage{amsmath}
\usepackage{euscript}
\usepackage{bbm}
\usepackage{tikz}
\usetikzlibrary{hobby,backgrounds,patterns}

\def\cal H{{\mathcal H}}

\def\R{\mathbb{R}}

\def\dom{{\text{\rm dom\,}}}
\def\phi{\varphi}
\def\d{\textup{d}}

\DeclareMathOperator{\diver}{div}

\renewcommand{\theta}{\vartheta}

\newtheorem{theorem}{Theorem}[section]
\newtheorem*{thm*}{Theorem}
\newtheorem{proposition}[theorem]{Proposition}
\newtheorem{corollary}[theorem]{Corollary}
\newtheorem{lemma}[theorem]{Lemma}

\theoremstyle{definition}

\newtheorem{remark}[theorem]{Remark}

\numberwithin{equation}{section}

\title[]{A remark on the order of mixed Dirichlet--Neumann eigenvalues of polygons}

\author[J.~Rohleder]{Jonathan Rohleder}
\address{Matematiska institutionen\\ Stockholms universitet \\
106 91 Stockholm \\
Sweden}
\email{jonathan.rohleder@math.su.se}

\begin{document}

\begin{abstract}
Given the Laplacian on a planar, convex domain with piecewise linear boundary subject to mixed Dirichlet--Neumann boundary conditions, we provide a sufficient condition for its lowest eigenvalue to dominate the lowest eigenvalue of the Laplacian with the complementary boundary conditions (i.e.\ with Dirichlet replaced by Neumann and vice versa). The application of this result to triangles gives an affirmative partial answer to a recent conjecture. Moreover, we prove a further observation of similar flavor for right triangles.
\end{abstract}

\maketitle

\section{Introduction}

We consider the Laplacian $- \Delta_\Gamma$ on a bounded, convex polygon $\Omega \subset \R^2$ subject to a Dirichlet boundary condition on a part $\Gamma$ of the boundary and a Neumann boundary condition on the complement $\Gamma^c$. The operator $- \Delta_\Gamma$ is self-adjoint and has a purely discrete spectrum, and its lowest eigenvalue $\lambda_1^\Gamma$ is positive, provided $\Gamma$ is nonempty. It is clear that enlarging $\Gamma$ leads to an increase of $\lambda_1^\Gamma$, but making a different choice of $\Gamma$ with the same or a larger length may in some cases lead to a smaller value of $\lambda_1^\Gamma$, that is, $\lambda_1^\Gamma$ does not depend monotonously on the size of $\Gamma$. 

The present note provides two observations on monotonicity properties of the lowest eigenvalue with respect to the choice of $\Gamma$, and it is inspired by recent results for triangles and other special domains in~\cite{S16}, see also the survey~\cite{LS17}. In the first result of this note, Theorem~\ref{thm:split}, we compare $\lambda_1^\Gamma$ to the lowest eigenvalue $\lambda_1^{\Gamma^c}$ of the mixed Laplacian $- \Delta_{\Gamma^c}$ satisfying the complementary boundary conditions, i.e., a Dirichlet condition on $\Gamma^c$ and a Neumann condition on $\Gamma$. We show that the inequality
\begin{align*}
 \lambda_1^{\Gamma^c} \leq \lambda_1^\Gamma
\end{align*}
holds if $\Gamma^c$ consists of one single side of the polygon $\Omega$ and the two angles where $\Gamma$ and $\Gamma^c$ meet are both strictly smaller than $\pi/2$; see Figure~\ref{fig:examples} for examples.
\begin{figure}[h]
\begin{tikzpicture}
\pgfsetlinewidth{0.8pt}
\color{gray}
\pgfputat{\pgfxy(-3,-1.2)}{\pgfbox[center,base]{$\Gamma^c$}}
\pgfxyline(-0.7,-0.8)(-5,-0.8)
\color{black}
\pgfxyline(-5,-0.8)(-4,0.3)
\pgfxyline(-4,0.3)(-2,0.7)
\pgfxyline(-2,0.7)(-1,0.5)
\pgfxyline(-1,0.5)(-0.7,-0.8)
\pgfputat{\pgfxy(-2.5,-0.3)}{\pgfbox[center,base]{$\Omega$}}
\pgfputat{\pgfxy(-4.3,0.3)}{\pgfbox[center,base]{$\Gamma$}}
\end{tikzpicture} 
\qquad \qquad \qquad
\begin{tikzpicture}
\pgfsetlinewidth{0.8pt}
\color{gray}
\pgfputat{\pgfxy(-3.2,0.1)}{\pgfbox[center,base]{$\Gamma^c = L$}}
\pgfxyline(-3,-0.8)(-4.8,1)
\color{black}
\pgfxyline(-5,-0.8)(-3,-0.8)
\pgfxyline(-4.8,1)(-5,-0.8)
\pgfputat{\pgfxy(-4.3,-0.3)}{\pgfbox[center,base]{$\Omega$}}
\pgfputat{\pgfxy(-5.25,-0.9)}{\pgfbox[center,base]{$\Gamma$}}
\end{tikzpicture}
\vspace*{2mm}
\caption{Two settings for which $\lambda_1^{\Gamma^c} \leq \lambda_1^\Gamma$ holds. In each case, $\Gamma$ and $\Gamma^c$ are drawn in black and gray, respectively.}
\label{fig:examples}
\end{figure}
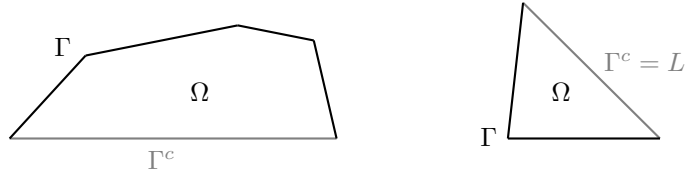
This result can be used to affirm a part of a conjecture on the lowest eigenvalues of triangles raised by Siudeja in~\cite[Conjecture~1.2]{S16}. In fact, for an arbitrary triangle whose sides we denote by $S, M$ and $L$, ordered nondecreasingly by their lengths, it follows
\begin{align*}
 \lambda_1^L \leq \lambda_1^{M \cup S},
\end{align*}
which, to the best of our knowledge, was known before only for certain classes of right triangles. 

The second result of this note, Theorem~\ref{thm:triangle}, is more restrictive and applies only to right triangles. It complements the recent results in~\cite{S16} by stating that
\begin{align*}
 \max \{\lambda_1^S, \lambda_1^M\} \leq \lambda_1^L
\end{align*}
holds for any right triangle, i.e., imposing the Dirichlet condition on the hypotenuse always leads to a larger (or equal) lowest eigenvalue than having the Dirichlet condition on one of the catheti. 

The proofs of Theorem~\ref{thm:split} and Theorem~\ref{thm:triangle} rely on plugging a certain partial derivative of an eigenfunction into the Rayleigh quotient and using an integral identity for the second partial derivatives of Sobolev functions on polygons. A similar approach was used for the comparison of mixed and Dirichlet Laplacian eigenvalues on polygons and polyhedra in~\cite{LR17}, see also~\cite{LW86}.

Let us finally mention that properties of eigenvalues of the Laplacian with mixed boundary conditions on special polygons have played an important role in various contexts. For instance, they were used in the famous construction of isospectral domains in~\cite{GWW92} and, more recently, in connection with the hot spots conjecture in~\cite{S15}.

\section{Preliminaries}

Let us set the stage and collect a few ingredients for the proofs of our main results. Recall first that for $\Gamma$ being any choice of sides of the polygon $\Omega$ the (negative) Laplacian $- \Delta_\Gamma$ can be defined as the self-adjoint operator in $L^2 (\Omega)$ which corresponds to the semibounded, closed quadratic form
\begin{align*}
 H_{0, \Gamma}^1 (\Omega) := \left\{ u \in H^1 (\Omega) : u |_\Gamma = 0 \right\} \ni u \mapsto \int_\Omega |\nabla u|^2 \d x.
\end{align*}
The functions in the domain of $- \Delta_\Gamma$ satisfy a Dirichlet boundary condition on $\Gamma$ and a Neumann boundary condition (in a weak sense, see, e.g.,~\cite[Lemma~4.3]{M00} for a definition of the weak Neumann trace) on the complement $\Gamma^c = \partial \Omega \setminus \Gamma$. If $u \in \dom (- \Delta_\Gamma)$ is sufficiently regular (see Proposition~\ref{prop:regularity} below) then the Neumann condition on $\Gamma^c$  can be interpreted in the usual sense, requiring the trace of $\nabla u \cdot \nu$ to vanish on $\Gamma^c$, where $\nu$ is the outer unit normal field on the boundary.

The operator $- \Delta_\Gamma$ has a compact resolvent and its lowest eigenvalue $\lambda_1^\Gamma$ is positive provided $\Gamma$ is nonempty. It is nondegenerate and can be expressed by the variational identity
\begin{align}\label{eq:minmax}
 \lambda_1^\Gamma = \min_{u \in H_{0, \Gamma}^1 (\Omega)} \frac{\int_\Omega |\nabla u|^2 \d x}{\int_\Omega |u|^2 \d x}.
\end{align}

Below we make use of the following regularity result which follows from~\cite[Theorem~4.4.3.3 and Lemma~4.4.1.4]{G85}.

\begin{proposition}\label{prop:regularity}
Assume that all angles at which $\Gamma$ and $\Gamma^c$ meet are strictly less than $\pi/2$. Then $\dom (- \Delta_\Gamma) \subset H^2 (\Omega)$.
\end{proposition}

Moreover, we will significantly make use of the following identity, which is a consequence of~\cite[Lemma~4.3.1.1--4.3.1.3]{G85}. 

\begin{lemma}\label{lem:Grisvard}
Let $u \in H^2 (\Omega)$ satisfy a Dirichlet boundary condition on $\Gamma$ and a Neumann boundary condition on its complement $\Gamma^c$. Then
\begin{align*}
 \int_\Omega (\partial_{1 2} u)^2 \d x = \int_\Omega (\partial_{1 1} u) (\partial_{2 2} u) \d x.
\end{align*}
\end{lemma}

We emphasize that the latter statement is valid for polygons only and fails for more general, curved domains.

\section{An ordering result for the lowest mixed eigenvalues of polygons}

In this section we prove the following first result of this note.

\begin{theorem}\label{thm:split}
Let $\Omega \subset \R^2$ be a polygon and let $\Gamma \subset \partial \Omega$ such that $\Gamma^c = \partial \Omega \setminus \Gamma$ consists of one single line segment. Moreover, suppose that the angles at both vertices where $\Gamma$ and $\Gamma^c$ meet are strictly less than $\pi/2$. Then
\begin{align*}
 \lambda_1^{\Gamma^c} \leq \lambda_1^\Gamma.
\end{align*}
\end{theorem}

\begin{proof}
Without loss of generality we assume that $\Gamma^c$ is parallel to the $x_2$-axis. Let $u$ be a real-valued eigenfunction of $- \Delta_\Gamma$ corresponding to the eigenvalue $\lambda_1^\Gamma$, and let $v = \partial_1 u$. It follows from Proposition~\ref{prop:regularity} that $v$ belongs to $H^1 (\Omega)$. Moreover, since $u$ satisfies a Neumann boundary condition on $\Gamma^c$ and the first unit vector $(1, 0)^\top$ is normal to $\Gamma^c$, it follows $v |_{\Gamma^c} = 0$, i.e., $v$ is an admissible test function for the Rayleigh quotient of $- \Delta_{\Gamma^c}$. Note that $v$ is nontrivial since $\partial_1 u = 0$ identically on $\Omega$ together with $u |_\Gamma = 0$ would imply $u = 0$ on $\Omega$. We employ Lemma~\ref{lem:Grisvard} to obtain
\begin{align}\label{eq:yeah}
\begin{split}
 0 & = \int_\Omega (\partial_{11} u) (\partial_{22} u) - (\partial_{12} u)^2 \d x = \int_\Omega \big( (\partial_{11} u) \Delta u - (\partial_{11} u)^2 - (\partial_{12} u)^2 \big) \d x \\
 & = - \lambda_1^\Gamma \int_\Omega \diver \binom{\partial_1 u}{0} u \d x - \int_\Omega \big( (\partial_{11} u)^2 + (\partial_{12}u)^2 \big) \d x \\
 & = \lambda_1^\Gamma \Big( \int_\Omega \binom{\partial_1 u}{0} \cdot \nabla u \d x - \int_{\partial \Omega} u \binom{\partial_1 u}{0} \cdot \nu \d \sigma \Big) - \int_\Omega |\nabla (\partial_1 u)|^2 \d x \\
 & = \lambda_1^\Gamma \int_\Omega (\partial_1 u)^2 \d x - \int_\Omega |\nabla (\partial_1 u)|^2 \d x,
 \end{split}
\end{align}
where in the last step we have used $u |_\Gamma = 0$ and $\partial_1 u |_{\Gamma^c} = 0$. It follows
\begin{align}\label{eq:voila}
 \int_\Omega |\nabla v|^2 \d x = \lambda_1^\Gamma \int_\Omega v^2 \d x,
\end{align}
and the assertion of the theorem follows with the help of the identity~\eqref{eq:minmax} applied to $\Gamma^c$ instead of $\Gamma$.
\end{proof}

\begin{remark}
The idea of using derivatives of eigenfunctions as test functions was used in~\cite{LW86} to establish eigenvalue inequalities between Dirichlet and Neumann eigenvalues of the Laplacian on smooth domains. In~\cite{LR17} it was used to compare mixed and Dirichlet eigenvalues on polygons and polyhedra. We remark that the methods of~\cite{LR17} may be employed to extend Theorem~\ref{thm:split} to higher dimensions.
\end{remark}

Next we apply Theorem~\ref{thm:split} to triangles and obtain the following three statements. If $\Omega$ is a triangle we denote its sides by $S$, $M$ and $L$, in nondecreasing order of their lengths. We remark that the inequality (iii) in the following corollary is a part of Conjecture~1.2 in~\cite{S16}.

\begin{corollary}
If $\Omega$ is any triangle then the following assertions hold.
\begin{enumerate}
 \item If both angles enclosing $S$ are strictly less than $\pi/2$ then $\lambda_1^S \leq \lambda_1^{L \cup M}$.
 \item If both angles enclosing $M$ are strictly less than $\pi/2$ then $\lambda_1^M \leq \lambda_1^{L \cup S}$.
 \item In any case, $\lambda_1^L \leq \lambda_1^{M \cup S}$.
\end{enumerate}
\end{corollary}

\begin{proof}
The assertions (i) and (ii) are direct consequences of Theorem~\ref{thm:split}. For item~(iii) just note that the angles enclosing the longest edge of a triangle can never be equal to or larger than $\pi/2$.
\end{proof}

\section{A remark on the lowest eigenvalues of right triangles}

In this short section we restrict ourselves to the class of right triangles and compare the lowest eigenvalue for a Dirichlet condition on a cathetus to the one for the hypotenuse, see Figure~\ref{fig:rightTriangle}. This extends observations from~\cite[Theorem~1.1]{S16} (where additional restrictions on the angles were required) to right triangles with arbitrary angles.

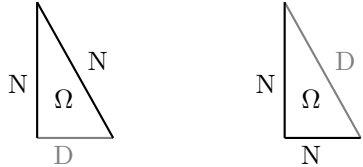
\begin{figure}[h]
\begin{tikzpicture}
\pgfsetlinewidth{0.8pt}
\color{gray}
\pgfxyline(-5,-0.8)(-4,-0.8)
\pgfputat{\pgfxy(-4.65,-1.15)}{\pgfbox[center,base]{D}}
\color{black}
\pgfxyline(-5,1)(-4,-0.8)
\pgfxyline(-5,-0.8)(-5,1)
\pgfputat{\pgfxy(-4.65,-0.4)}{\pgfbox[center,base]{$\Omega$}}
\pgfputat{\pgfxy(-5.25,-0.2)}{\pgfbox[center,base]{N}}
\pgfputat{\pgfxy(-4.2,0.1)}{\pgfbox[center,base]{N}}
\end{tikzpicture} 
\qquad \qquad \qquad
\begin{tikzpicture}
\pgfsetlinewidth{0.8pt}
\color{gray}
\pgfxyline(-5,1)(-4,-0.8)
\pgfputat{\pgfxy(-4.2,0.1)}{\pgfbox[center,base]{D}}
\color{black}
\pgfxyline(-5,-0.8)(-4,-0.8)
\pgfxyline(-5,-0.8)(-5,1)
\pgfputat{\pgfxy(-4.65,-0.4)}{\pgfbox[center,base]{$\Omega$}}
\pgfputat{\pgfxy(-5.25,-0.2)}{\pgfbox[center,base]{N}}
\pgfputat{\pgfxy(-4.65,-1.15)}{\pgfbox[center,base]{N}}
\end{tikzpicture}
\vspace*{2mm}
\caption{Two choices of mixed boundary conditions on the same right triangle (D = Dirichlet, N = Neumann). By Theorem~\ref{thm:triangle} the lowest eigenvalue of the left configuration does not exceed the lowest eigenvalue of the right one.}
\label{fig:rightTriangle}
\end{figure}

\begin{theorem}\label{thm:triangle}
Let $\Omega$ be a right triangle with sides $S, M$ and $L$ ordered nondecreasingly by their lengths. Then 
\begin{align*}
 \max \{ \lambda_1^S, \lambda_1^M\} \leq \lambda_1^L.
\end{align*}
\end{theorem}

\begin{proof}
We show the inequality $\lambda_1^S \leq \lambda_1^L$; the inequality $\lambda_1^M \leq \lambda_1^L$ is analogous. The proof is similar to the proof of Theorem~\ref{thm:split}. Let us assume w.l.o.g.\ that $S$ is parallel to the $x_2$-axis. We take a nontrivial, real-valued $u \in \ker (- \Delta_L - \lambda_1^L)$ and set $v = \partial_1 u$. Then $v$ is nontrivial and $v |_S = 0$. Now we repeat the calculation~\ref{eq:yeah} with $\Gamma = L$ and observe that the boundary integral is zero since $u$ vanishes on $L$, which is the hypotenuse, $\partial_1 u$ vanishes on $S$, and for the other cathetus, $M$, the normal vector $\nu$ is plus or minus the second unit vector $(0, 1)^\top$. Hence we arrive at~\eqref{eq:voila} with $\Gamma = L$, which completes the proof.
\end{proof}

\end{document}